\documentclass[11pt]{article}
\usepackage{amsfonts, amsmath, amssymb, latexsym, eucal, amscd,cite} 
\usepackage[all]{xy}
\newtheorem{theorem}{Theorem}[section]
\newtheorem{lemma}[theorem]{Lemma}
\newtheorem{proposition}[theorem]{Proposition}
\newtheorem{corollary}[theorem]{Corollary}
\newtheorem{exAux}[theorem]{Example}

\newtheorem{Def}[theorem]{Definition}
\newenvironment{definition}{\begin{Def} \rm}{\end{Def}}
\newtheorem{Note}[theorem]{Note}
\newenvironment{note}{\begin{Note} \rm}{\end{Note}}
\newtheorem{Problem}[theorem]{Problem}

\newtheorem{Rem}[theorem]{Remark}

\newtheorem{Not}[theorem]{Notation}

\newtheorem{Conj}[theorem]{Conjecture}

\newtheorem{Ass}[theorem]{Assumption}

\newenvironment{proof}{\medskip\noindent{\bf Proof.\ }}{\qed\medskip}
\newenvironment{proofof}[1]{\medskip\noindent{\bf Proof  of {#1}.\ 
}}{\qed\medskip}
\newcommand{\qed}{\hfill\mbox{$\Box$\qquad\qquad}}

\newcommand{\F}{\mathbb{F}}

\newcommand{\vphi}{\varphi}
\renewcommand{\th}{\theta}

\newcommand{\vth}{\vartheta}

\newcommand{\tth}{\tilde{\theta}}
\newcommand{\tvphi}{\tilde{\varphi}}
\newcommand{\tphi}{\tilde{\phi}}
\newcommand{\ta}{\tilde{a}}

\newif\ifDRAFT
\DRAFTfalse

\begin{document}

\begin{center}
\LARGE\bf
Leonard pairs having specified end-entries
\end{center}

\begin{center}
\Large
Kazumasa Nomura
\end{center}

\bigskip

{\small
\begin{quote}
\begin{center}
{\bf Abstract}
\end{center}
Fix an algebraically closed field $\F$ and an integer $d \geq 3$.
Let $V$ be a vector space over $\F$ with dimension $d+1$.
A Leonard pair on $V$ is an ordered pair of diagonalizable linear transformations
$A: V \to V$ and $A^* : V \to V$,
each acting in an irreducible tridiagonal fashion on 
an eigenbasis for the other one.
Let $\{v_i\}_{i=0}^d$ (resp.\ $\{v^*_i\}_{i=0}^d$) be such an eigenbasis for $A$ 
(resp.\  $A^*$).
For $0 \leq i \leq d$ define a linear transformation $E_i : V \to V$ 
such that $E_i v_i=v_i$ and $E_i v_j =0$ if $j \neq i$ $(0 \leq j \leq d)$.
Define $E^*_i : V \to V$ in a similar way.
The sequence $\Phi =(A, \{E_i\}_{i=0}^d, A^*, \{E^*_i\}_{i=0}^d)$
is called a Leonard system on $V$ with diameter $d$.
With respect to the basis $\{v_i\}_{i=0}^d$,
let $\{\th_i\}_{i=0}^d$ (resp.\ $\{a^*_i\}_{i=0}^d$) be the diagonal entries
of the matrix representing $A$ (resp.\ $A^*$).
With respect to the basis $\{v^*_i\}_{i=0}^d$,
let $\{\th^*_i\}_{i=0}^d$ (resp.\ $\{a_i\}_{i=0}^d$) be the diagonal entries
of the matrix representing $A^*$ (resp.\ $A$).
It is known that $\{\th_i\}_{i=0}^d$ (resp. $\{\th^*_i\}_{i=0}^d$) 
are mutually distinct,
and the expressions
$(\th_{i-1}-\th_{i+2})/(\th_i-\th_{i+1})$,
$(\th^*_{i-1}-\th^*_{i+2})/(\th^*_i - \th^*_{i+1})$
are equal and independent of $i$ for $1 \leq i \leq d-2$.
Write this common value as $\beta + 1$.
In the present paper we consider the ``end-entries''
$\th_0$, $\th_d$, $\th^*_0$, $\th^*_d$, $a_0$, $a_d$, $a^*_0$, $a^*_d$.
We prove that a Leonard system with diameter $d$ 
is determined up to isomorphism by its end-entries and $\beta$
if and only if either
(i) $\beta \neq \pm 2$ and $q^{d-1} \neq -1$,
where $\beta=q+q^{-1}$,
or
(ii) $\beta = \pm 2$ and $\text{Char}(\F) \neq 2$.
\end{quote}
}

\section{Introduction}

Throughout the paper $\F$ denotes an algebraically closed field.

We begin by recalling the notion of a Leonard pair.
We use the following terms.
A square matrix is said to be {\em tridiagonal} whenever each nonzero
entry lies on either the diagonal, the subdiagonal, or the superdiagonal.
A tridiagonal matrix is said to be {\em irreducible} whenever
each entry on the subdiagonal is nonzero and each entry on the superdiagonal is nonzero.

\begin{definition}  {\rm (See \cite[Definition 1.1]{T:Leonard}.)}    \label{def:LP}   \samepage
\ifDRAFT {\rm def:LP}. \fi
Let $V$ be a vector space over $\F$ with finite positive dimension.
By a {\em Leonard pair on $V$} we mean an ordered pair of linear transformations
$A : V \to V$ and $A^*: V \to V$ that satisfy (i) and (ii) below:
\begin{itemize}
\item[\rm (i)]
There exists a basis for $V$ with respect to which the matrix representing
$A$ is irreducible tridiagonal and the matrix representing $A^*$ is diagonal.
\item[\rm (ii)]
There exists a basis for $V$ with respect to which the matrix representing
$A^*$ is irreducible tridiagonal and the matrix representing $A$ is diagonal.
\end{itemize}
\end{definition}

\begin{note}    \samepage
According to a common notational convention, $A^*$ denotes the
conjugate transpose of $A$.
We are not using this convention.
In a Leonard pair $A,A^*$ the matrices $A$ and $A^*$ are arbitrary subject to
(i) and (ii) above.
\end{note}

We refer the reader to \cite{NT:balanced,T:Leonard,T:24points,T:array,T:survey}
for background on Leonard pairs.

For the rest of this section, fix an integer $d \geq 0$ and a vector space $V$
over $\F$ with dimension $d+1$.
Consider a Leonard pair $A,A^*$ on $V$.
By \cite[Lemma 1.3]{T:Leonard} each of $A,A^*$ has mutually distinct $d+1$ eigenvalues.
Let $\{\th_i\}_{i=0}^d$ (resp.\ $\{\th^*_i\}_{i=0}^d$) be an ordering of the eigenvalues of $A$ (resp.\ $A^*$),
and let $\{V_i\}_{i=0}^d$ (resp.\ $\{V^*_i\}_{i=0}^d$) be the corresponding eigenspaces.
For $0 \leq i \leq d$ define a linear transformation $E_i : V \to V$ such that
$(E_i - I) V_i = 0$ and $E_i V_j=0$ for $j \neq i$ $(0 \leq j \leq d)$.
Here $I$ denotes the identity.
We call $E_i$ the {\em primitive idempotent} of $A$ associated with $\th_i$.
The primitive idempotent $E^*_i$ of $A^*$ associated with $\th^*_i$
is similarly defined.
For $0 \leq i \leq d$ pick a nonzero $v_i \in V_i$.
We say the ordering $\{E_i\}_{i=0}^d$ is {\em standard} whenever
the basis $\{v_i\}_{i=0}^d$ satisfies Definition \ref{def:LP}(ii).
A standard ordering of the primitive idempotents of $A^*$ is
similarly defined.
For a standard ordering $\{E_i\}_{i=0}^d$, the ordering $\{E_{d-i}\}_{i=0}^d$
is also standard and no further ordering is standard.
Similar result applies to a standard ordering of the primitive idempotents
of $A^*$.

\begin{definition}   {\rm (See \cite[Definition 1.4]{T:Leonard}.) }    \label{def:LS}
\ifDRAFT {\rm def:LS}. \fi 
By a {\em Leonard system} on $V$ we mean a sequence
\[
 \Phi = (A, \{E_i\}_{i=0}^d, A^*, \{E^*_i\}_{i=0}^d),
\]
where $A,A^*$ is a Leonard pair on $V$, and 
$\{E_i\}_{i=0}^d$ (resp.\ $\{E^*_i\}_{i=0}^d$) is a standard ordering 
of the primitive idempotents of $A$ (resp.\ $A^*$).
We call $d$ the {\em diameter} of $\Phi$.
We say {\em $\Phi$ is over $\F$}.
\end{definition}

We recall the notion of an isomorphism of Leonard systems.
Consider a Leonard system $\Phi = (A, \{E_i\}_{i=0}^d, A^*, \{E^*_i\}_{i=0}^d)$
on $V$ and a Leonard system 
$\Phi' = (A', \{E'_i\}_{i=0}^d, A^{*\prime}, \{E^{* \prime}_i\}_{i=0}^d)$
on a vector space $V'$ with dimension $d+1$.
By an {\em isomorphism of Leonard systems from $\Phi$ to $\Phi'$} we mean
a linear bijection $\sigma : V \to V'$
such that $\sigma A = A' \sigma$, $\sigma A^* = A^{*\prime} \sigma$,
and $\sigma E_i = E'_i \sigma$, $\sigma E^*_i = E^{*\prime} \sigma$ for $0 \leq i \leq d$.
Leonard systems $\Phi$ and $\Phi'$ are said to be {\em isomorphic}
whenever there exists an isomorphism of Leonard systems from
$\Phi$ to $\Phi'$.

We recall the parameter array of a Leonard system.

\begin{definition}   {\rm (See \cite[Section 2]{T:array}, \cite[Theorem 4.6]{NT:formula}.) }
\label{def:parray}
\ifDRAFT {\rm def:parray}. \fi
Let $\Phi =  (A, \{E_i\}_{i=0}^d, A^*, \{E^*_i\}_{i=0}^d)$
be a Leonard system over $\F$.
By the {\em parameter array of $\Phi$} we mean the sequence
\begin{equation}            \label{eq:parray}
   (\{\th_i\}_{i=0}^d, \{\th^*_i\}_{i=0}^d, \{\vphi_i\}_{i=1}^d, \{\phi_i\}_{i=1}^d),
\end{equation}
where $\th_i$ is the eigenvalue of $A$ associated with $E_i$,
$\th^*_i$ is the eigenvalue of $A^*$ associate with $E^*_i$,
and
\begin{align*}
 \vphi_i &=  (\th^*_0 - \th^*_i)
       \frac{\text{tr}(E^*_0 \prod_{h=0}^{i-1}(A-\th_h I))}
               {\text{tr}(E^*_0\prod_{h=0}^{i-2}(A-\th_h I))},
\\
 \phi_i &= (\th^*_0 - \th^*_i)
       \frac{\text{tr}(E^*_0\prod_{h=0}^{i-1}(A-\theta_{d-h}I))}
              {\text{tr}(E^*_0\prod_{h=0}^{i-2}(A-\theta_{d-h}I))}, 
\end{align*}
where tr means trace.
In the above expressions, the denominators are nonzero by 
\cite[Corollary 4.5]{NT:formula}.
\end{definition}

\begin{lemma}    {\rm (See  \cite[Theorem 1.9]{T:Leonard}.) }
\label{lem:characterize}   \samepage
\ifDRAFT {\rm lem:characterize}. \fi
A Leonard system is determined up to isomorphism by its
parameter array.
\end{lemma}

\begin{lemma}  {\rm (See \cite[Theorem 1.9]{T:Leonard}.) } \label{lem:classify}
\ifDRAFT {\rm lem:classify}. \fi
Consider a sequence \eqref{eq:parray} consisting of scalars taken from $\F$.
Then there exists a Leonard system over $\F$ with parameter array 
\eqref{eq:parray} if and only if {\rm (i)--(v)} hold below:
\begin{itemize}
\item[\rm (i)]  
$\th_i \neq \th_j$, $\;\; \th^*_i \neq \th^*_j\;\;$
    $\;\;(0 \leq i < j \leq d)$.
\item[\rm (ii)]
$\vphi_i \neq 0$, $\;\; \phi_i \neq 0\;\;$ $\;\; (1 \leq i \leq d)$.
\item[\rm (iii)]
$ \displaystyle
 \vphi_i = \phi_1 \sum_{\ell=0}^{i-1} 
                            \frac{\th_\ell - \th_{d-\ell}}
                                   {\th_0 - \th_d}
             + (\th^*_i - \th^*_0)(\th_{i-1} - \th_d)  \qquad (1 \leq i \leq d).
$
\item[\rm (iv)]
$ \displaystyle
 \phi_i = \vphi_1 \sum_{\ell=0}^{i-1} 
                            \frac{\th_\ell - \th_{d-\ell}}
                                   {\th_0 - \th_d}
              + (\th^*_i - \th^*_0)(\th_{d-i+1} - \th_0)  \qquad (1 \leq i \leq d).
$
\item[\rm (v)]
The expressions
\begin{equation}           \label{eq:indep}
   \frac{\th_{i-2} - \th_{i+1}}
          {\th_{i-1}-\th_{i}},
 \qquad\qquad
   \frac{\th^*_{i-2} - \th^*_{i+1}}
          {\th^*_{i-1} - \th^*_{i}}
\end{equation}
are equal and independent of $i$ for $2 \leq i \leq d-1$.
\end{itemize}
\end{lemma}

\begin{definition}   \label{def:parrayF}   \samepage
\ifDRAFT {\rm def:parrayF}. \fi
By a {\em parameter array over $\F$} we mean a sequence \eqref{eq:parray}
consisting of scalars taken from $\F$ that satisfy conditions (i)--(v)
in Lemma \ref{lem:classify}.
\end{definition}

\begin{definition}    \label{def:beta}    \samepage
\ifDRAFT {\rm def:beta}. \fi
Let 
$\Phi = (A, \{E_i\}_{i=0}^d, A^*, \{E^*_i\}_{i=0}^d)$
be a Leonard system over $\F$ with diameter $d \geq 3$.
Let \eqref{eq:parray} be the parameter array of $\Phi$.
By the {\em fundamental parameter} of $\Phi$ we mean
one less than the common value of \eqref{eq:indep}.
\end{definition}

Let $\Phi =  (A, \{E_i\}_{i=0}^d, A^*, \{E^*_i\}_{i=0}^d)$
be a Leonard system over $\F$
with parameter array \eqref{eq:parray}.
In \cite{N:endparam} we considered the {\em end-parameters}:
\[
 \th_0, \quad
 \th_d, \quad
 \th^*_0, \quad
 \th^*_d, \quad
 \vphi_1, \quad
 \vphi_d, \quad
 \phi_1, \quad
 \phi_d.
\]
We proved that a Leonard system with diameter $d \geq 3$ is determined 
up to isomorphism by its fundamental parameter and its end-parameters 
(see Lemma \ref{lem:endparam}).
In the present paper we consider another set of parameters.

\begin{definition} {\rm \cite[Definition 2.5]{T:Leonard} }   \label{def:ai}   \samepage
\ifDRAFT {\rm def:ai}. \fi
For $0 \leq i \leq d$ define
\[
  a_i = \text{\rm tr}(A E^*_i),   \qquad\qquad
  a^*_ i = \text{\rm tr}(A^* E_i).
\]
We call $\{a_i\}_{i=0}^d$ (resp.\ $\{a^*_i\}_{i=0}^d$) the
{\em principal sequence} (resp.\ {\em dual principal sequence}) of $\Phi$.
\end{definition}

The principal sequence and the dual principal sequence have 
the following geometric interpretation.
For $0 \leq i \leq d$ pick a nonzero $v_i \in E_i V$ and
a nonzero $v^*_i \in E^*_i V$.
Note that each of $\{v_i\}_{i=0}^d$ and $\{v^*_i\}_{i=0}^d$ is a basis for $V$.
As easily observed,
with respect to the basis $\{v_i\}_{i=0}^d$ the matrix representing $A$
has diagonal entries $\{\th_i\}_{i=0}^d$ and the matrix representing $A^*$
has diagonal entries $\{a^*_i\}_{i=0}^d$.
Similarly, with respect to the basis $\{v^*_i\}_{i=0}^d$ the matrix representing $A^*$
has diagonal entries $\{\th^*_i\}_{i=0}^d$ and the matrix representing $A$
has diagonal entries $\{a_i\}_{i=0}^d$. 

We now state our main results.
Let  
$\Phi = (A, \{E_i\}_{i=0}^d, A^*, \{E^*_i\}_{i=0}^d)$
be a Leonard system over $\F$ with parameter array \eqref{eq:parray}.
Let $\{a_i\}_{i=0}^d$ (resp. $\{a^*_i\}_{i=0}^d$) be the principal sequence
(resp.\ dual principal sequence) of $\Phi$.
We consider the {\em end-entries}:
\[
\th_0, \quad
\th_d, \quad
\th^*_0, \quad
\th^*_d, \quad
a_0, \quad
a_d, \quad
a^*_0, \quad
a^*_d.
\]
The end-entries are algebraically dependent: 

\begin{proposition}   \label{prop:restriction}    \samepage
\ifDRAFT {\rm prop:restriction}. \fi
Assume $d \geq 1$.
Then
\begin{equation}     \label{eq:restriction}
 \frac{(a_0-\th_0)(a_d-\th_d)}
        {(a_0-\th_d)(a_d-\th_0)}
= \frac{(a^*_0-\th^*_0)(a^*_d-\th^*_d)}
         {(a^*_0-\th^*_d)(a^*_d-\th^*_0)}.
\end{equation}  
In \eqref{eq:restriction}, all the denominators are nonzero.
Once we fix $\th_0$, $\th_d$, $\th^*_0$, $\th^*_d$, 
any one of $a_0$, $a_d$, $a^*_0$, $a^*_d$ is determined 
by the remaining three.
\end{proposition}

For the rest of this section, assume $d \geq 3$.

\begin{theorem}   \label{thm:main1}   \samepage
\ifDRAFT {\rm thm:main1}. \fi
Assume $\vphi_1 + \vphi_d \neq \phi_1 + \phi_d$.
Let $\Phi'$ be a Leonard system over $\F$ with diameter $d$ 
that has the same fundamental parameter and the same end-entries as $\Phi$.
Then $\Phi'$ is isomorphic to $\Phi$.
\end{theorem}

In Appendix A, we display formulas that represent the parameter array
in terms of the fundamental parameter and  the end-entries.

\begin{theorem}  \label{thm:main2}     \samepage
\ifDRAFT {\rm thm:main2}. \fi
Assume $\vphi_1 + \vphi_d = \phi_1 + \phi_d$.
Then there exist infinitely many mutually non-isomorphic Leonard systems over $\F$
with diameter $d$
that have the same fundamental parameter and the same end-entries as $\Phi$.
\end{theorem}

In our proof of Theorem \ref{thm:main2}, we 
construct infinitely many Leonard systems that has the same fundamental
parameter and the same end-entries as $\Phi$.
The condition $\vphi_1 + \vphi_d = \phi_1 + \phi_d$ is interpreted in terms of 
the fundamental parameter as follows:

\begin{proposition}    \label{prop:Delta0}    \samepage
\ifDRAFT {\rm prop:Delta0}. \fi
Let $\beta$ be the fundamental parameter of $\Phi$, and
pick a nonzero $q \in \F$ such that $\beta = q+q^{-1}$.
Then $\vphi_1+\vphi_d = \phi_1+\phi_d$ if and only if one of the following {\rm (i), (ii)}
holds:
\begin{itemize}
\item[\rm (i)]
$\beta \neq \pm 2$ and $q^{d-1} = -1$.
\item[\rm (ii)]
$\beta = 0$ and $\text{\rm Char}(\F)=2$.
\end{itemize}
\end{proposition}

By Theorems \ref{thm:main1}, \ref{thm:main2} and Proposition \ref{prop:Delta0}
we obtain:

\begin{corollary}  \label{cor:main}   \samepage
\ifDRAFT {\rm cor:main}. \fi
A Leonard system with diameter $d$ is determined up to isomorphism by its
fundamental parameter $\beta$ and its end-entries if and only if
one of the following {\rm (i)}, {\rm (ii)} holds:
\begin{itemize}
\item[\rm (i)]
$\beta \neq \pm 2$ and $q^{d-1} \neq -1$, where $\beta = q+q^{-1}$.
\item[\rm (ii)]
$\beta = \pm 2$ and $\text{\rm Char}(\F) \neq 2$.
\end{itemize}
\end{corollary}

The paper is organized as follows.
In Section \ref{sec:D4} we recall the action of the dihedral group $D_4$
on the set of all Leonard systems.
In Section \ref{sec:endentries} we collect some formulas concerning
end-entries.
In Section \ref{sec:restriction} we prove Proposition \ref{prop:restriction}.
In Section \ref{sec:types} we recall the notion of the type of a Leonard system.
In Section \ref{sec:endparam} we recall some results from \cite{N:endparam},
and prove Proposition \ref{prop:Delta0}.
In Section \ref{sec:proof1} we prove Theorem \ref{thm:main1}.
In Section \ref{sec:Delta0} we prove a lemma for later use.
In Sections \ref{sec:proof2(i)} and \ref{sec:proof2(ii)} we prove Theorem \ref{thm:main2}.

\section{The $D_4$ action}
\label{sec:D4}

For a Leonard system 
$\Phi = (A, \{E_i\}_{i=0}^d, A^*, \{E^*_i\}_{i=0}^d)$ over $\F$,
each of the following is a Leonard system over $\F$:
\begin{align*}
\Phi^{*}  &:= (A^*, \{E^*_i\}_{i=0}^d, A, \{E_i\}_{i=0}^d), 
\\
\Phi^{\downarrow} &:= (A, \{E_i\}_{i=0}^d, A^*, \{E^*_{d-i}\}_{i=0}^d),
\\
\Phi^{\Downarrow} &:= (A, \{E_{d-i}\}_{i=0}^d, A^*, \{E^*_{i}\}_{i=0}^d).
\end{align*}
Viewing $*$, $\downarrow$, $\Downarrow$ as permutations on the set of
all Leonard systems,
\begin{equation}    \label{eq:relation}
*^2 =\,\, \downarrow^2 \,\, = \,\, \Downarrow^2 \, = 1,   \qquad
\Downarrow * \, = \, * \downarrow,  \qquad
\downarrow * \, = \, * \Downarrow,  \qquad
\downarrow \Downarrow \,\, = \,\, \Downarrow \downarrow.
\end{equation}
The group generated by symbols $*$, $\downarrow$, $\Downarrow$ subject
to the relations \eqref{eq:relation} is the dihedral group $D_4$. 
We recall $D_4$ is the group of symmetries of a square, and has $8$ elements.
For an element $g \in D_4$, and for an object $f$ associated with $\Phi$, 
let $f^g$ denote the corresponding object associated with $\Phi^{g^{-1}}$.
The $D_4$ action affects the parameter array as follows:

\begin{lemma}   {\rm (See \cite[Theorem 1.11]{T:Leonard}). }  \label{lem:D4}  \samepage
\ifDRAFT {\rm lem:D4}. \fi
Let $\Phi$ be a Leonard system over $\F$ with parameter array 
$(\{\th_i\}_{i=0}^d, \{\th^*_i\}_{i=0}^d, \{\vphi_i\}_{i=1}^d, \{\phi_i\}_{i=1}^d)$.
Then for $g \in \{\downarrow, \Downarrow, *\}$ the scalars
$\th_i^g$, ${\th^*_i}^g$, $\vphi^g_i$, $\phi^g_i$
are as follows:
\[
\begin{array}{c|cccc}
 g \; & \quad \th^g_i & \quad {\th^*_i}^g & \quad  \vphi^g_i & \; \phi^g_i
\\ \hline
 \downarrow \; & \quad \th_i & \quad  \th^*_{d-i} &  \quad  \phi_{d-i+1} & \;  \vphi_{d-i+1}
 \rule{0mm}{4mm}
\\
 \Downarrow \;  &\quad \th_{d-i} & \quad  \th^*_i &  \quad  \phi_i & \;  \vphi_i
\\
 * \; & \quad \th^*_i &  \quad \th_i &  \quad  \vphi_i &  \; \phi_{d-i+1} 
\end{array}
\] 
\end{lemma}

\begin{lemma}    \label{lem:aiD4}    \samepage
\ifDRAFT {\rm lem:aiD4}. \fi
Let $\Phi$ be a Leonard system over $\F$ with principal sequence $\{a_i\}_{i=0}^d$
and dual principal sequence $\{a^*_i\}_{i=0}^d$.
Then for $0 \leq i \leq d$
\begin{align*}
 a_i^{\downarrow} &= a_{d-i}, &
 {a^*_i}^{\downarrow} &= a^*_i,   &
 a_i^{\Downarrow} &= a_i,   &
 {a^*_i}^{\Downarrow} &= a^*_{d-i}.
\end{align*}
\end{lemma}

\begin{proof}
Immediate from Definition \ref{def:ai}.
\end{proof}

\section{The end-entries}
\label{sec:endentries}

In this section we recall some formulas concerning the end-entries.
Fix an integer $d \geq 1$.
Let 
$\Phi = (A, \{E_i\}_{i=0}^d$, $A^*, \{E^*_i\}_{i=0}^d)$
be a Leonard system over $\F$ with parameter array
$(\{\th_i\}_{i=0}^d, \{\th^*_i\}_{i=0}^d, \{\vphi_i\}_{i=1}^d$, $\{\phi_i\}_{i=1}^d)$.
Let $\{a_i\}_{i=0}^d$ (resp.\ $\{a^*_i\}_{i=0}^d$) be the principal sequence
(resp. dual principal sequence) of $\Phi$.

\begin{lemma}   {\rm (See \cite[Lemma 5.1]{T:Leonard}, \cite[Lemma 10.3]{T:24points}.)}
\label{lem:a0th0}   \samepage
\ifDRAFT {\rm lem:a0th0}. \fi
With the above notation,
\begin{align}
 a_0 - \th_0 &= \frac{\vphi_1}{\th^*_0 - \th^*_1},  &
 a_d - \th_d &= \frac{\vphi_d}{\th^*_d - \th^*_{d-1}},    \label{eq:a0ad1}
\\
 a_0 - \th_d &= \frac{\phi_1}{\th^*_0 - \th^*_1}, &
 a_d - \th_0 &= \frac{\phi_d}{\th^*_{d}- \th^*_{d-1}},     \label{eq:a0ad2}
\\
 a^*_0 - \th^*_0 &= \frac{\vphi_1}{\th_0 - \th_1},   &
 a^*_d - \th^*_d &=  \frac{\vphi_d}{\th_d - \th_{d-1}},    \label{eq:as0asd1}
\\
 a^*_0 - \th^*_d &=  \frac{\phi_d}{\th_0 - \th_1}, &
 a^*_d - \th^*_0 &= \frac{\phi_1}{\th_d - \th_{d-1}}.      \label{eq:as0asd2}
\end{align}
\end{lemma}

\begin{note}
In Lemma \ref{lem:a0th0}, the eight equations are obtained from one of them
by applying $D_4$.
\end{note}

The following lemmas are well-known.
We give a short proof based on Lemma \ref{lem:a0th0}
for convenience of the reader.

\begin{lemma}    \label{lem:a0neqth0}    \samepage
\ifDRAFT {\rm lem:a0neqth0}. \fi
With the above notation,
\begin{align*}
 a_0 &\neq \th_0,  &  
 a_0 &\neq \th_d,  &  
 a_d &\neq \th_0,  & 
 a_d &\neq \th_d,
\\
 a^*_0 &\neq \th^*_0,  &
 a^*_0 &\neq \th^*_d, &
 a^*_d &\neq \th^*_0, &
 a^*_d & \neq \th^*_d.
\end{align*}
\end{lemma}

\begin{proof}
Follows from Lemma \ref{lem:classify}(ii) and Lemma \ref{lem:a0th0}.
\end{proof} 

\begin{lemma}   \label{lem:a0th0/a0thd}    \samepage
\ifDRAFT {\rm lem:a0th0/a0thd}. \fi
With the above notation,
\begin{align}
\frac{a_0-\th_0}{a_0-\th_d} &= \frac{\vphi_1}{\phi_1},  &
\frac{a_d-\th_d}{a_d-\th_0} &= \frac{\vphi_d}{\phi_d},         \label{eq:a0th0/a0thd}
\\
\frac{a^*_0-\th^*_0}{a^*_0-\th^*_d} &= \frac{\vphi_1}{\phi_d}, &
\frac{a^*_d-\th^*_d}{a^*_d-\th^*_0} &= \frac{\vphi_d}{\phi_1}.  \label{eq:as0ths0/as0thsd}
\end{align}
\end{lemma}

\begin{proof}
Follows from Lemma \ref{lem:a0th0}.
\end{proof}

\begin{lemma}    \label{lem:th1ths1}   \samepage
\ifDRAFT {\rm lem:th1ths1}. \fi
With the above notation,
\begin{align}
 \th_0 - \th_1 &= \frac{\phi_d - \vphi_1}{\th^*_0 - \th^*_d}, 
&
\th_d - \th_{d-1} &= \frac{\vphi_d - \phi_1}{\th^*_0 - \th^*_d},  \label{eq:th1thd-1}
\\
\th^*_0 - \th^*_1 &= \frac{\phi_1 - \vphi_1}{\th_0 - \th_d}, 
&
\th^*_d - \th^*_{d-1} &= \frac{\vphi_d - \phi_d}{\th_0-\th_d}.      \label{eq:ths1thsd-1}
\end{align}
\end{lemma}

\begin{proof}
The equation on the left in \eqref{eq:th1thd-1} is obtained from the equations
on the left in \eqref{eq:as0asd1} and \eqref{eq:as0asd2}.
The remaining equations can be obtained in a similar way.
\end{proof}

\begin{note}
Lemma \ref{lem:th1ths1} can be obtained also from
Lemma \ref{lem:classify}(iii), (iv).
\end{note}

\begin{lemma}   \label{lem:vphi-phi}    \samepage
\ifDRAFT {\rm lem:vphi-phi}. \fi
With the above notation,
\[
  \phi_1 \neq \vphi_1, \qquad\qquad
  \phi_1 \neq \vphi_d, \qquad\qquad
  \phi_d \neq \vphi_1, \qquad\qquad
  \phi_d \neq \vphi_d.
\]
\end{lemma}

\begin{proof}
Follows from Lemma \ref{lem:th1ths1} and Lemma \ref{lem:classify}(i).
\end{proof}

\begin{lemma}    \label{lem:key}    \samepage
\ifDRAFT {\rm lem:key}. \fi
With the above notation,
\begin{align*}
 a_0 - \th_0 &= \frac{\vphi_1(\th_0-\th_d)}{\phi_1-\vphi_1},  &
 a_d - \th_d &= \frac{\vphi_d(\th_0-\th_d)}{\vphi_d-\phi_d},
\\
 a_0 - \th_d &= \frac{\phi_1(\th_0-\th_d)}{\phi_1 - \vphi_1}, &
 a_d - \th_0 &= \frac{\phi_d(\th_0-\th_d)}{\vphi_d-\phi_d},
\\
 a^*_0 - \th^*_0 &= \frac{\vphi_1(\th^*_0 - \th^*_d)}{\phi_d-\vphi_1},  &
 a^*_d - \th^*_d &= \frac{ \vphi_d (\th^*_0 - \th^*_d) }  
                                   { \vphi_d - \phi_1},
\\
 a^*_0 - \th^*_d &= \frac{\phi_d(\th^*_0 - \th^*_d)}{\phi_d - \vphi_1}, &
 a^*_d - \th^*_0 &= \frac{\phi_1(\th^*_0 - \th^*_d)}{\vphi_d - \phi_1}.
\end{align*}
\end{lemma}

\begin{proof}
Use Lemmas \ref{lem:a0th0} and \ref{lem:th1ths1}.
\end{proof}

\begin{lemma}   \label{lem:key2}   \samepage
\ifDRAFT {\rm lem:key2}. \fi
With the above notation,
\begin{align*}
 a_0 &= \frac{\th_0 \phi_1 - \th_d \vphi_1}{\phi_1 - \vphi_1}, 
&
 a_d &= \frac{\th_d \phi_d - \th_0 \vphi_d}{\phi_d - \vphi_d},
\\
 a^*_0 &= \frac{\th^*_0 \phi_d - \th^*_d \vphi_1}{\phi_d - \vphi_1},
&
 a^*_d &= \frac{\th^*_d \phi_1 - \th^*_0 \vphi_d}{\phi_1 - \vphi_d}. 
\end{align*}
\end{lemma}

\begin{proof}
Follows from Lemma \ref{lem:key}.
\end{proof}

\section{Proof of Proposition \ref{prop:restriction}}
\label{sec:restriction}

In this section we prove Proposition \ref{prop:restriction}.
Fix an integer $d \geq 1$.
Let
$\Phi = (A, \{E_i\}_{i=0}^d, A^*$, $\{E^*_i\}_{i=0}^d)$
be a Leonard system over $\F$ with parameter array
$(\{\th_i\}_{i=0}^d, \{\th^*_i\}_{i=0}^d, \{\vphi_i\}_{i=1}^d$, $\{\phi_i\}_{i=1}^d)$.
Let $\{a_i\}_{i=0}^d$ (resp.\ $\{a^*_i\}_{i=0}^d$) be the principal sequence
(resp.\ dual principal sequence) of $\Phi$.

\begin{proofof}{Proposition \ref{prop:restriction}}
In \eqref{eq:restriction}, all the denominators are nonzero by Lemma \ref{lem:a0neqth0}.
Using Lemma \ref{lem:a0th0/a0thd} one checks that each side of \eqref{eq:restriction}
is equal to $\vphi_1 \vphi_d (\phi_1 \phi_d)^{-1}$,
so \eqref{eq:restriction} holds. 
Rewrite \eqref{eq:restriction} as
\begin{equation}       \label{eq:restriction2}
 (a_0-\th_0)(a_d-\th_d)(a^*_0-\th^*_d)(a^*_d-\th^*_0)
 - (a_0-\th_d)(a_d-\th_0)(a^*_0-\th^*_0)(a^*_d-\th^*_d) = 0.
\end{equation}
Viewing \eqref{eq:restriction2} as a linear equation in $a_0$, the coefficient of
$a_0$ is
\[
 (a_d-\th_d)(a^*_0-\th^*_d)(a^*_d-\th^*_0) 
 - (a_d-\th_0)(a^*_0-\th^*_0)(a^*_d-\th^*_d).
\]
Using Lemma \ref{lem:key} one checks that the above coefficient is equal to
\[
   \frac{\vphi_d \phi_d (\th_0 - \th_d)(\th^*_0 - \th^*_d)^2 (\phi_1 - \vphi_1)}
          {(\phi_1-\vphi_d)(\phi_d-\vphi_1)(\phi_d-\vphi_d)}.
\]
This is nonzero by Lemma \ref{lem:classify}(i), (ii) and Lemma \ref{lem:vphi-phi}.
So one can solve \eqref{eq:restriction2} in $a_0$.
Thus $a_0$ is determined by 
$\th_0$, $\th_d$, $\th^*_0$, $\th^*_d$, $a_d$, $a^*_0$, $a^*_d$.
Concerning $a_d$, $a^*_0$, $a^*_d$, apply the above arguments to
$\Phi^{\downarrow}$, $\Phi^*$, $\Phi^{\downarrow *}$.
\end{proofof}

\section{The type of a Leonard system}
\label{sec:types}

In this section we recall the type of a Leonard system.
Let $\Phi$ be a Leonard system over $\F$ with diameter $d \geq 3$.
Let $\beta$ be the fundamental parameter of $\Phi$,
and pick a nonzero $q \in \F$ such that $\beta = q+q^{-1}$.

\begin{definition}      \label{def:types}    \samepage
We define the {\em type} of $\Phi$ as follows:
\[
 \begin{array}{c|lll}
  \text{Type of $\Phi$} & & \text{Description} 
\\ \hline
  \text{I} & \;\; \beta \neq  2, & \; \beta \neq -2 \rule{0mm}{4.5mm}
\\
 \text{II} & \;\; \beta=2, & \text{Char}(\F) \neq 2
\\
\text{III}^+ & \;\; \beta = -2, & \text{Char}(\F) \neq 2, & \text{$d$ is even}
\\
\text{III}^- & \;\; \beta = -2, &  \text{Char}(\F) \neq 2, & \text{$d$ is odd}
\\
\text{IV} & \;\; \beta = 0, & \text{Char}(\F)=2
 \end{array}
\]
\end{definition}

\begin{lemma}  {\rm (See \cite[Sections 13--17]{NT:affine}.) }
\label{lem:types}   \samepage
\ifDRAFT {\rm lem:types}. \fi
The following hold:
\begin{itemize}
\item[\rm (i)]
Assume $\Phi$ has type I.
Then $q^i \neq 1$ for $1 \leq i \leq d$.
\item[\rm (ii)]
Assume $\Phi$ has type II.
Then $\text{\rm Char}(\F) \neq i$ for any prime $i$ such that $i \leq d$.
\item[\rm (iii)]
Assume $\Phi$ has type III$^+$.
Then $\text{\rm Char}(\F) \neq i$ for any prime $i$ such that $i \leq d/2$.
In particular, neither of $d$, $d-2$ vanish in $\F$.
\item[\rm (iv)]
Assume $\Phi$ has type III$^-$.
Then $\text{\rm Char}(\F) \neq i$ for any prime $i$ such that $i \leq (d-1)/2$.
In particular, $d-1$ does not vanish in $\F$.
\item[\rm (v)]
Assume $\Phi$ has type IV.
Then $d=3$.
\end{itemize}
\end{lemma}

\section{The end-parameters}
\label{sec:endparam}

In this section, we first recall some results from \cite{N:endparam}.
We then prove Proposition \ref{prop:Delta0}.
Fix an integer $d \geq 3$.
Let $\Phi$ be a Leonard system over $\F$ with diameter $d$
that has parameter array
$(\{\th_i\}_{i=0}^d, \{\th^*_i\}_{i=0}^d, \{\vphi_i\}_{i=1}^d, \{\phi_i\}_{i=1}^d)$.
Consider the end-parameters:
\[
\th_0, \quad
\th_d, \quad
\th^*_0, \quad
\th^*_d, \quad
\vphi_1, \quad
\vphi_d, \quad
\phi_1, \quad
\phi_d.
\]

\begin{lemma}  {\rm (See \cite[Theorem 1.9]{N:endparam}.)}
 \label{lem:endparam}   \samepage
\ifDRAFT {\rm lem:endparam}. \fi
A Leonard system over $\F$ with diameter $d$ is determined up to isomorphism 
by its fundamental parameter and its end-parameters.
\end{lemma}

We recall a relation between the end-parameters and the fundamental
parameter.

\begin{lemma}   {\rm (See \cite[Proposition 1.11]{N:endparam}.)}
\label{lem:Omega}   \samepage
\ifDRAFT {\rm lem:Omega}. \fi
Let $\beta$ be the fundamental parameter of $\Phi$,
and pick a nonzero $q \in \F$ such that $\beta = q+q^{-1}$.
Then the scalar
\begin{equation}      \label{eq:defOmega}
 \Omega = \frac{\phi_1 + \phi_d - \vphi_1 - \vphi_d}
                      {(\th_0-\th_d)(\th^*_0 - \th^*_d)}
\end{equation}
is as follows:
\[
\begin{array}{c|c}
 \text{\rm Type of $\Phi$} & \Omega 
\\ \hline
\text{\rm I} 
 & \displaystyle  \frac{ (q-1)(q^{d-1}+1)}{q^d-1}
  \rule{0mm}{7mm}
\\
\text{\rm II}
 &   2/d   \rule{0mm}{5mm}
\\
\text{\rm III$^+$}
&  2(d-1)/d   \rule{0mm}{5mm}
\\
\text{\rm III$^-$}
  &  2     \rule{0mm}{5mm}
\\
\text{\rm IV}
&  0     \rule{0mm}{5mm}
\end{array}
\]
\end{lemma}

\begin{lemma}    \label{lem:Omega0pre}  \samepage
\ifDRAFT {\rm lem:Omega0pre}. \fi
With reference to Lemma \ref{lem:Omega},
the following {\rm (i)} and {\rm (ii)} are equivalent:
\begin{itemize}
\item[\rm (i)]
$\Omega = 0$.
\item[\rm (ii)]
$\vphi_1 + \vphi_d = \phi_1 + \phi_d$.
\end{itemize}
\end{lemma}

\begin{proof}
Immediate from \eqref{eq:defOmega}.
\end{proof}

\begin{lemma}    \label{lem:Omega0}   \samepage
\ifDRAFT {\rm lem:Omega0}. \fi
With reference to Lemma \ref{lem:Omega},
$\Omega=0$ if and only if one of the following {\rm (i), (ii)} holds:
\begin{itemize}
\item[\rm (i)]
$\beta \neq \pm 2$ and $q^{d-1}=-1$.
\item[\rm (ii)]
$\beta = 0$ and $\text{\rm Char}(\F)=2$.
\end{itemize}
\end{lemma}

\begin{proof}
First assume $\beta \neq \pm 2$.
Then $\Phi$ has type I.
Now by Lemma \ref{lem:Omega} $\Omega = 0$ if and only if
$q^{d-1}+1 = 0$.
Next assume $\beta = \pm 2$.
Then $\Phi$ has one of types  II, III$^+$, III$^-$, IV.
By Lemma \ref{lem:Omega}, $\Omega \neq 0$ for types II, III$^+$, III$^-$,
and $\Omega=0$ for type IV.
The result follows.
\end{proof}

\begin{proofof}{Proposition \ref{prop:Delta0}}
Follows from Lemmas \ref{lem:Omega0pre} and \ref{lem:Omega0}.
\end{proofof}

\section{Proof of Theorem \ref{thm:main1}}
\label{sec:proof1}

Fix an integer $d \geq 3$.
Let $\Phi$ be a Leonard system over $\F$ with diameter $d$
that has parameter array
$ (\{\th_i\}_{i=0}^d, \{\th^*_i\}_{i=0}^d, \{\vphi_i\}_{i=1}^d, \{\phi_i\}_{i=1}^d)$.
Let $\{a_i\}_{i=0}^d$ (resp.\ $\{a^*_i\}_{i=0}^d$) be the principal sequence 
(resp.\ dual principal sequence) of $\Phi$.
Let the scalar $\Omega$ be from \eqref{eq:defOmega}.
Define 
\begin{equation}         \label{eq:Delta}
 \Delta = (a_0 - \th_0)(a^*_0 - \th^*_d) - (a_d - \th_0)(a^*_0 - \th^*_0).
\end{equation}

\begin{lemma}    \label{lem:Delta}    \samepage
\ifDRAFT {\rm lem:Delta}. \fi
With the above notation,
\begin{equation}            \label{eq:Delta2}
 \Delta = 
  \frac{\vphi_1 \phi_d (\th_0-\th_d)(\th^*_0-\th^*_d)(\phi_1+\phi_d-\vphi_1-\vphi_d)}
         {(\phi_1-\vphi_1)(\phi_d-\vphi_1)(\phi_d-\vphi_d)}.   
\end{equation}
\end{lemma}

\begin{proof}
Routine verification using Lemma \ref{lem:key}.
\end{proof}

\begin{lemma}   \label{lem:Delta0}   \samepage
\ifDRAFT {\rm lem:Delta0}. \fi
With the above notation, the following {\rm (i)--(ii)} are equivalent:
\begin{itemize}
\item[\rm (i)]
$\Delta=0$.
\item[\rm (ii)]
$\vphi_1 + \vphi_d = \phi_1 + \phi_d$.
\end{itemize}
\end{lemma}

\begin{proof}
Follows from  \eqref{eq:Delta2} and Lemma \ref{lem:classify}(i), (ii).
\end{proof}

Consider the following expressions:
\begin{align*}
 \Gamma_1 &= (a_0-\th_0)(a_d-\th_0)(a^*_0-\th^*_0) (\th^*_0-\th^*_d),
\\
 \Gamma_2 &= (a_0-\th_0)(a_d-\th_d)(a^*_0-\th^*_d) (\th^*_0-\th^*_d),
\\
 \Gamma_3 &= (a_0-\th_d)(a_d-\th_0)(a^*_0-\th^*_0) (\th^*_0-\th^*_d), 
\\
 \Gamma_4 &= (a_0-\th_0)(a_d-\th_0)(a^*_0-\th^*_d) (\th^*_0-\th^*_d). 
\end{align*}

\begin{lemma}    \label{lem:DelGa}    \samepage
\ifDRAFT {\rm lem:DelGa}. \fi
With the above notation,
\begin{align}
 \Gamma_1 &= 
 - \frac{\vphi_1^2 \phi_d (\th_0-\th_d)^2 (\th^*_0-\th^*_d)^2}
           {(\phi_1-\vphi_1)(\phi_d-\vphi_1)(\phi_d-\vphi_d)},      \notag
\\
 \Gamma_2 &= 
 - \frac{\vphi_1 \vphi_d \phi_d (\th_0-\th_d)^2 (\th^*_0-\th^*_d)^2}
           {(\phi_1-\vphi_1)(\phi_d-\vphi_1)(\phi_d-\vphi_d)},    \notag
\\
 \Gamma_3 &=
 - \frac{\vphi_1 \phi_1 \phi_d (\th_0-\th_d)^2 (\th^*_0-\th^*_d)^2}
           {(\phi_1-\vphi_1)(\phi_d-\vphi_1)(\phi_d-\vphi_d)},       \notag
\\
 \Gamma_4 &=
 - \frac{\vphi_1\phi_d^2 (\th_0-\th_d)^2 (\th^*_0-\th^*_d)^2}
           {(\phi_1-\vphi_1)(\phi_d-\vphi_1)(\phi_d-\vphi_d)}.     \notag
\end{align}
\end{lemma}

\begin{proof}
Routine verification using Lemma \ref{lem:key}.
\end{proof}

\begin{lemma}     \label{lem:DeltaGamma}    \samepage
\ifDRAFT {\rm lem:DeltaGamma}. \fi
With the above notation,
\[
          \frac{\Omega \Gamma_1}{\vphi_1}
          = \frac{\Omega \Gamma_2}{\vphi_d}
          = \frac{\Omega \Gamma_3}{\phi_1}
          = \frac{\Omega \Gamma_4}{\phi_d}
          = - \Delta.
\]
\end{lemma}

\begin{proof}
Follows from \eqref{eq:defOmega}, \eqref{eq:Delta2} and Lemma \ref{lem:DelGa}.
\end{proof}

\begin{lemma}    \label{lem:vphiphi}   \samepage
\ifDRAFT {\rm lem:vphiphi}. \fi
With the above notation, assume $\Delta \neq 0$.
Then
\begin{align*}
 \vphi_1 &= - \frac{\Omega \Gamma_1}{\Delta},  &
 \vphi_d &= - \frac{\Omega \Gamma_2}{\Delta},  &
 \phi_1  &= - \frac{\Omega \Gamma_3}{\Delta},  &
 \phi_d  &= - \frac{\Omega \Gamma_4}{\Delta}.
\end{align*}
\end{lemma}

\begin{proof}
Immediate from Lemma \ref{lem:DeltaGamma}.
\end{proof}

\begin{proofof}{Theorem \ref{thm:main1}}
Assume $\vphi_1 + \vphi_d \neq \phi_1 + \phi_d$.
Note that $\Delta \neq 0$ by Lemma \ref{lem:Delta0}.
Observe that each of $\Delta$, $\Gamma_1$, $\Gamma_2$, $\Gamma_3$, $\Gamma_4$
is determined by the end-entries.
By this and Lemma \ref{lem:vphiphi} each of $\vphi_1$, $\vphi_d$, $\phi_1$, $\phi_d$
is determined by the end-entries and $\Omega$.
By Lemma \ref{lem:Omega} $\Omega$ is determined by the fundamental
parameter $\beta$.
By these comments, the end-parameters are determined by the end-entries and
$\beta$.
Now the result follows by Lemma \ref{lem:endparam}.
\end{proofof}

\section{A lemma}
\label{sec:Delta0}

Fix an integer $d \geq 3$.
Let $\Phi$ be a Leonard system over $\F$ with diameter $d$
that has parameter array
$ (\{\th_i\}_{i=0}^d, \{\th^*_i\}_{i=0}^d$, $\{\vphi_i\}_{i=1}^d, \{\phi_i\}_{i=1}^d)$.
Let $\{a_i\}_{i=0}^d$ (resp.\ $\{a^*_i\}_{i=0}^d$) be the principal sequence 
(resp.\ dual principal sequence) of $\Phi$.

\begin{lemma}    \label{lem:Delta0relations}    \samepage
\ifDRAFT {\rm lem:Delta0relations}. \fi
Assume $\vphi_1+\vphi_d = \phi_1 + \phi_d$. Then
\begin{align}
 \frac{a_d-\th_0}{a_0-\th_0} &= \frac{a^*_0-\th^*_d}{a^*_0-\th^*_0},     \label{eq:a0th0/adth0}
\\
 \frac{a_0-\th_d}{a_0-\th_0} &= \frac{a^*_d-\th^*_0}{a^*_0-\th^*_0},    \label{eq:a0thd/a0th0}
\\
 \frac{a_d-\th_d}{a_0-\th_0} &= \frac{a^*_d-\th^*_d}{a^*_0-\th^*_0},   \label{eq:a0th0/adthd}
\\
 \frac{a_0-a_d}{a_0-\th_0} &= \frac{\th^*_d-\th^*_0}{a^*_0-\th^*_0}.   \label{eq:a0ad/a0th0}
\end{align}
\end{lemma}

\begin{proof}
Let the scalar $\Delta$ be from \eqref{eq:Delta}.
By Lemma \ref{lem:Delta0} $\Delta=0$, and \eqref{eq:a0th0/adth0} follows.
Line \eqref{eq:a0thd/a0th0} is obtained from \eqref{eq:a0th0/adth0} by 
applying $*$.
Combining \eqref{eq:restriction} and \eqref{eq:a0th0/adth0},
\[
   \frac{a_d - \th_d}{a_0 - \th_d} = \frac{a^*_d - \th^*_d}{a^*_d - \th^*_0}.
\]
By this and \eqref{eq:a0thd/a0th0} we get \eqref{eq:a0th0/adthd}.
By \eqref{eq:a0th0/adth0} minus \eqref{eq:a0th0/adthd},
\[
     \frac{\th_d - \th_0}{a_0-\th_0} = \frac{a^*_0-a^*_d}{a^*_0-\th^*_0}.
\]
Applying $*$ to this, we obtain \eqref{eq:a0ad/a0th0}.
\end{proof}

\section{Proof of Theorem \ref{thm:main2}; part 1}
\label{sec:proof2(i)}

In this and the next section we prove Theorem \ref{thm:main2}.
Fix an integer $d \geq 3$.
Let $\Phi$ be a Leonard system over $\F$ with diameter $d$, and let
$(\{\th_i\}_{i=0}^d, \{\th^*_i\}_{i=0}^d, \{\vphi_i\}_{i=1}^d, \{\phi_i\}_{i=1}^d)$
be the parameter array of $\Phi$.
Let $\beta$ be the fundamental parameter of $\Phi$,
and pick a nonzero $q \in \F$ such that $\beta=q+q^{-1}$.
Let $\{a_i\}_{i=0}^d$ (resp.\ $\{a^*_i\}_{i=0}^d$) be the principal sequence (resp.\ duall
principal sequence) of $\Phi$.
Assume $\vphi_1 + \vphi_d = \phi_1 + \phi_d$.
By Proposition \ref{prop:Delta0} we have two cases:
\begin{itemize}
\item[] Case (i): 
$\Phi$ has type I and $q^{d-1}=-1$.
\item[] Case (ii):
$\Phi$ has type IV.
\end{itemize}
In this section we consider Case (i).
Note that $q^i \neq 1$ for $1 \leq i \leq d$ by Lemma \ref{lem:types}(i).
Also note that $\text{Char}(\F) \neq 2$;
otherwise $q^{d-1}=1$.

For a nonzero scalar $\zeta \in \F$ we define a sequence 
\[
  \tilde{p}(\zeta) =
  (\{\tth_i(\zeta)\}_{i=0}^d, \{\tth^*_i(\zeta)\}_{i=0}^d, 
\{\tvphi_i(\zeta) \}_{i=1}^d, \{\tphi_i(\zeta)\}_{i=1}^d)
\]
as follows.
For $0 \leq i \leq d$ define
\begin{align}
 L_i (\zeta) &= 
  (q+1)(q^{i-1}+1) \zeta - (q-1)(q^{i-1}-1)(a^*_0-\th^*_0)(\th_0-\th_d),   \label{eq:Li}                                  
\\
 K_i (\zeta) &= - \frac{q^i-1} {2 q^{i-1} (q^2-1)(a^*_0-\th^*_0)} L_i (\zeta).  \label{eq:Ki}
\end{align}
For $0 \leq i \leq d$ define $L^*_i (\zeta)$ and $K^*_i (\zeta)$ by viewing $\zeta^* = \zeta$:
\begin{align}
 L^*_i (\zeta) &=
   (q+1)(q^{i-1}+1) \zeta - (q-1)(q^{i-1}-1)(a_0-\th_0)(\th^*_0 - \th^*_d),  \label{eq:Lsi}
\\
 K^*_i (\zeta) &=
   - \frac{q^i-1}{2 q^{i-1} (q^2-1)(a_0-\th_0)} \, L^*_i (\zeta).                  \label{eq:Ksi}
\end{align}
For $0 \leq i \leq d$ define $L^\Downarrow_i (\zeta)$ and $K^\Downarrow_i (\zeta)$
by viewing $\zeta^\Downarrow = (a_0-\th_d) \zeta / (a_0-\th_0)$:
\begin{align}
 L^\Downarrow_i (\zeta) &= 
   \frac{(q+1)(q^{i-1}+1)(a_0-\th_d)}
          {a_0 - \th_0}  \, \zeta 
  + (q-1)(q^{i-1}-1)(a^*_d-\th^*_0)(\th_0-\th_d),                                \label{eq:LDi}
\\
 K^\Downarrow_i (\zeta) 
   &= - \frac{q^i-1} {2 q^{i-1} (q^2-1)(a^*_d-\th^*_0)} L^\Downarrow_i (\zeta).   \label{eq:KDi}
\end{align}
Now for $0 \leq i \leq d$ define
\begin{align}
 \tth_i (\zeta) &= \th_0 + K_i (\zeta),  &
 \tth^*_i (\zeta) &= \th^*_0 + K^*_i (\zeta),                              \label{eq:typeItthitthsi}
\end{align}
and for $1 \leq i \leq d$ define
\begin{align}
 \tvphi_i (\zeta) 
  &= K^{\Downarrow}_{d-i+1} (\zeta) K^*_i (\zeta)
     + \frac{(q^i-1)(q^{i-2}+1) (a_0-\th_d)} {q^{i-2} (q^2-1) (a_0-\th_0)} \, \zeta,   \label{eq:typeItvphi}
\\
 \phi_i (\zeta)
  &= K_{d-i+1}(\zeta) K^*_i (\zeta)
      + \frac{(q^i-1) (q^{i-2}+1)}{q^{i-2}(q^2-1)} \, \zeta.                         \label{eq:typeItphi}
\end{align}
We have defined $\tilde{p}(\zeta)$.
We prove the following result:

\begin{proposition}   \label{prop:typeI}   \samepage
\ifDRAFT {\rm prop:typeI}. \fi
The sequence $\tilde{p}(\zeta)$ is a parameter array  over $\F$
for infinitely many values of $\zeta$.
Assume $\tilde{p} (\zeta)$ is the parameter array of a Leonard system $\tilde{\Phi}(\zeta)$.
Then $\tilde{\Phi}(\zeta)$ has the same fundamental parameter and the
same end-entries as $\Phi$.
\end{proposition}

The following three lemmas can be routinely verified.

\begin{lemma}    \label{lem:thcoincide}   \samepage
\ifDRAFT {\rm lem:thcoincide}. \fi
We have
\begin{align*} 
 \tth_0(\zeta) &= \th_0, &
 \tth_d(\zeta)  &= \th_d, &
 \tth^*_0(\zeta)  &= \th^*_0, &
 \tth^*_d(\zeta) &= \th^*_d.
\end{align*}
\end{lemma}

\begin{lemma}   \label{lem:typeItthitthj}    \samepage
\ifDRAFT {\rm lem:typeItthitthj}. \fi
For $0 \leq i,j \leq d$
\begin{align}
\tth_i (\zeta) - \tth_j (\zeta) &=
 \frac{q^j-q^i }{2 q^{i+j-1} }
   \left( \frac{q^{i+j-1}+1}{(q-1)(a^*_0-\th^*_0)} \, \zeta
           - \frac{(q^{i+j-1}-1)(\th_0-\th_d)}{q+1} \right),   \label{eq:typeIthithj}
\\
\tth^*_i (\zeta) - \tth^*_j (\zeta) &=
 \frac{q^j-q^i }{2 q^{i+j-1} }
   \left( \frac{q^{i+j-1}+1}{(q-1)(a_0-\th_0)} \, \zeta
           - \frac{(q^{i+j-1}-1)(\th^*_0-\th^*_d)}{q+1} \right).  \label{eq:typeIthsithsj}
\end{align}
\end{lemma}

\begin{lemma}   \label{lem:typeIindep}   \samepage
\ifDRAFT {\rm lem:typeIindep}. \fi
For $2 \leq i \leq d-1$
each of the expressions 
\[
  \frac{\tth_{i-2}(\zeta)-\tth_{i+1}(\zeta)}{\tth_{i-1}(\zeta)-\tth_i(\zeta)},
  \qquad\qquad
  \frac{\tth^*_{i-2}(\zeta)-\tth^*_{i+1}(\zeta)}{\tth^*_{i-1}(\zeta)-\tth^*_i(\zeta)}
\]
is equal to $q+q^{-1}+1$.
\end{lemma}

\begin{lemma}    \label{lem:typeIcond(i)}   \samepage
\ifDRAFT {\rm lem:typeIcond(i)}. \fi
For $0 \leq i < j \leq d$ the following hold:
\begin{itemize}
\item[\rm (i)]
$\tth_i (\zeta)=\tth_j (\zeta)$ holds for at most one value of $\zeta$.
\item[\rm (ii)]
$\tth^*_i (\zeta)=\tth^*_j (\zeta)$ holds for at most one value of $\zeta$.
\end{itemize}
\end{lemma}

\begin{proof}
(i):
First assume $q^{i+j-1} = -1$.
Then by \eqref{eq:typeIthithj}
\[
   \tth_i (\zeta) - \tth_j (\zeta) = \frac{(q^i-q^j)(\th_0-\th_d)}{q+1} \neq 0.
\]
Next assume $q^{i+j-1} \neq -1$.
Then by \eqref{eq:typeIthithj}, $\tth_i (\zeta) - \tth_j (\zeta)$
is a polynomial in $\zeta$ with degree $1$.
The result follows.

(ii):
Similar to the proof of (i).
\end{proof}

\begin{lemma}    \label{lem:tvphi1tphi1}   \samepage
\ifDRAFT {\rm lem:tvphi1tphi1}. \fi
We have
\begin{align*}
 \tvphi_1 (\zeta) &= \zeta,
&
 \tphi_1 (\zeta) &= \frac{a_0-\th_d}{a_0-\th_0} \, \zeta,
\\
 \tvphi_d (\zeta) &= \frac{a_d - \th_d}{a_0-\th_0} \, \zeta,
&
 \tphi_d (\zeta) &= \frac{a_d-\th_0}{a_0-\th_0} \, \zeta.
\end{align*}
\end{lemma}

\begin{proof}
Routine verification using $q^{d-1}=-1$.
\end{proof}

For $1 \leq i \leq d$ define
\[
 \vth_i = 
  \sum_{\ell=0}^{i-1}
    \frac{\tth_{\ell}(\zeta) - \tth_{d-\ell} (\zeta)}
           {\tth_0 (\zeta) - \tth_d (\zeta)}.
\]

\begin{lemma}    \label{lem:typeIvth}   \samepage
\ifDRAFT {\rm lem:typeIvth}. \fi
For $1 \leq i \leq d$
\begin{equation}                               \label{eq:typeIvth}
   \vth_i =  \frac{(q^i-1)(q^{i-2}+1)}{q^{i-2}(q^2-1)}.
\end{equation}
\end{lemma}

\begin{proof}
By \eqref{eq:typeIthithj},
\[
 \frac{\tth_\ell(\zeta)-\tth_{d-\ell}(\zeta)}{\tth_0(\zeta)-\tth_d(\zeta)}
 = \frac{q^{d-\ell} - q^{\ell}}{2q^{d-1}(\th_0-\th_d)}
   \left( \frac{q^{d-1}+1}{(q-1)(a^*_0-\th^*_0)}
           - \frac{(q^{d-1}-1)(\th_0-\th_d)}{q+1} \right).
\]
Simplify this using $q^{d-1}=-1$ to find
\[
  \frac{\tth_\ell(\zeta)-\tth_{d-\ell}(\zeta)}{\tth_0(\zeta)-\tth_d(\zeta)}
  = \frac{q^\ell + q^{1-\ell}}{q+1}.
\]
Now one routinely verifies \eqref{eq:typeIvth}.
\end{proof}

\begin{lemma}    \label{lem:KLinv}    \samepage
\ifDRAFT {\rm lem:KLinv}. \fi
For $1 \leq i \leq d$
\begin{align*}
 L_{d-i+1} (\zeta) &=
   q^{1-i}(q+1)(q^{i-1}-1) \zeta + q^{1-i}(q-1)(q^{i-1}+1)(a^*_0-\th^*_0)(\th_0-\th_d),
\\
 K_{d-i+1} (\zeta) &=
   - \frac{q (q^{i-2}+1)}{2(q^2-1)(a^*_0-\th^*_0)} \, L_{d-i+1} (\zeta).
\end{align*}
\end{lemma}

\begin{proof}
Follows from \eqref{eq:Li}, \eqref{eq:Ki} and $q^{d-1}=-1$.
\end{proof}

\begin{lemma}    \label{lem:KDLDinv}    \samepage
\ifDRAFT {\rm lem:KDLDinv}. \fi
For $1 \leq i \leq d$
\begin{align*}
 L^\Downarrow_{d-i+1} (\zeta) &=
   \frac{ (q+1)(q^{i-1} -1)(a_0-\th_d)} {q^{i-1} (a_0 - \th_0)} \, \zeta
 - q^{1-i} (q-1)(q^{i-1}+1)(a^*_d - \th^*_0)(\th_0 - \th_d),
\\
 K^\Downarrow_{d-i+1} (\zeta) &=
  - \frac{q(q^{i-2}+1)} {2 (q^2-1)(a^*_d - \th^*_0)} \, L^\Downarrow_{d-i+1} (\zeta).
\end{align*}
\end{lemma}

\begin{proof}
Follows from \eqref{eq:LDi}, \eqref{eq:KDi} and $q^{d-1}=-1$.
\end{proof}

\begin{lemma}   \label{lem:typeItvphitphi}   \samepage
\ifDRAFT {\rm lem:typeItvphitphi}. \fi
For $1 \leq i \leq d$
\begin{align*}
 \tvphi_i (\zeta) &= \tphi_1 (\zeta) \vth_i 
        + (\tth^*_i (\zeta) - \tth^*_0(\zeta))(\tth_{i-1} (\zeta)-\tth_d(\zeta)),
\\
 \tphi_i (\zeta) &= \tvphi_1 (\zeta) \vth_i 
       + (\tth^*_i (\zeta) - \tth^*_0(\zeta))(\tth_{d-i+1} (\zeta) - \tth_0(\zeta)).
\end{align*}
\end{lemma}

\begin{proof}
Routine verification using Lemmas \ref{lem:typeItthitthj} and \ref{lem:tvphi1tphi1}--\ref{lem:KDLDinv}.
\end{proof}

\begin{lemma}    \label{lem:typeIcond(ii)}   \samepage
\ifDRAFT {\rm lem:typeIcond(ii)}. \fi
For $1 \leq i \leq d$ the following hold:
\begin{itemize}
\item[\rm (i)]
$\tvphi_i (\zeta)=0$ holds for at most two values of $\zeta$.
\item[\rm (ii)]
$\tphi_i (\zeta)=0$ holds for at most two values of $\zeta$.
\end{itemize}
\end{lemma}

\begin{proof}
(i):
We show that $\tvphi_i (\zeta)$ is a polynomial in $\zeta$ with degree $1$ or $2$.
By Lemma \ref{lem:tvphi1tphi1} $\tvphi_1(\zeta)=\zeta$.
So we assume $2 \leq i \leq d$.
In view of \eqref{eq:typeItvphi}, first consider $K^*_i (\zeta)$.
Note that $q^{i-1}+1 \neq 0$; otherwise $q^{d-i} = q^{d-1} q^{1-i} = (-1)(-1)=1$.
By this and \eqref{eq:Lsi}, \eqref{eq:Ksi},  $K^*_i (\zeta)$ is a polynomial in $\zeta$ with degree $1$.
Next consider $K^\Downarrow_{d-i+1} (\zeta)$.
Note that $q^{i-2}+1 \neq 0$; otherwise $q^{d-i+1} = q^{d-1} q^{2-i}=(-1)(-1)=1$.
By this and Lemma \ref{lem:KDLDinv},
$K^\Downarrow_{d-i+1} (\zeta)$ is a polynomial in $\zeta$ with degree $1$.
By these comments and \eqref{eq:typeItvphi}, $\tvphi_i (\zeta)$ is a polynomial in $\zeta$
with degree $1$ or $2$.
The result follows.

(ii):
Similar to the proof of (i).
\end{proof}

\begin{lemma}    \label{lem:typeItparray}     \samepage
\ifDRAFT {\rm lem:typeItprray}. \fi
The sequence $\tilde{p}(\zeta)$ is a parameter array over $\F$
for infinitely many values of $\zeta$.  
\end{lemma}

\begin{proof}
We verify conditions (i)--(v) in Lemma \ref{lem:classify}.
Conditions (iii) and (iv) are satisfied by Lemma \ref{lem:typeItvphitphi}.
Condition (v) is satisfied by Lemma \ref{lem:typeIindep}.
Note that $\F$ has infinitely many elements since $\F$ is algebraically closed.
By this and Lemmas \ref{lem:typeIcond(i)}, \ref{lem:typeIcond(ii)},
conditions (i) and (ii) are satisfied for infinitely many values of $\zeta$.
The result follows.
\end{proof}

\begin{lemma}   \label{lem:vphidphid2}    \samepage
\ifDRAFT {\rm lem:vphidphid2}. \fi
We have
\begin{align*}
 \tvphi_d (\zeta) &=
   \frac{a^*_d-\th^*_d}{a^*_0-\th^*_0} \, \zeta,
&
 \tphi_1 (\zeta) &=
   \frac{a^*_d-\th^*_0}{a^*_0-\th^*_0} \, \zeta,  
&
 \tphi_d (\zeta) &=
   \frac{a^*_0-\th^*_d}{a^*_0-\th^*_0} \, \zeta.
\end{align*}
\end{lemma}

\begin{proof}
Follows from Lemma \ref{lem:tvphi1tphi1} using  
\eqref{eq:a0th0/adth0}--\eqref{eq:a0th0/adthd}.
\end{proof}

\begin{lemma}    \label{lem:typeIa0ad}   \samepage
\ifDRAFT {\rm lem:typeIa0ad}. \fi
Assume $\tilde{p}(\zeta)$ is a parameter array of a Leonard system $\tilde{\Phi}(\zeta)$.
Let $\{\ta_i (\zeta)\}_{i=0}^d$ (resp.\ $\{\ta^*_i (\zeta) \}_{i=0}^d$) be the principal
sequence (resp.\ dual principal sequence) of $\tilde{\Phi} (\zeta)$.
Then
\[
 \ta_0 (\zeta) = a_0,  \qquad
 \ta_d (\zeta) = a_d,  \qquad
 \ta^*_0 (\zeta) = a^*_0, \qquad
 \ta^*_d (\zeta) = a^*_d.
\]
\end{lemma}

\begin{proof}
Applying Lemma \ref{lem:key2} to $\tilde{\Phi}(\zeta)$ and
using Lemma \ref{lem:thcoincide},
\[
 \ta_0 (\zeta) = \frac{\th_0 \tphi_1(\zeta) - \th_d \tvphi_1(\zeta) }
                   {\tphi_1(\zeta) - \tvphi_1(\zeta)}.
\]
In this equation, eliminate $\tvphi_1(\zeta)$ and $\tphi_1(\zeta)$ using
Lemma \ref{lem:tvphi1tphi1} to get $\ta_0(\zeta) = a_0$.
The remaining three equations are obtained in a similar way
using Lemmas \ref{lem:tvphi1tphi1} and \ref{lem:vphidphid2}.
\end{proof}

\begin{proofof}{Proposition \ref{prop:typeI}}
By Lemma \ref{lem:typeItparray} $\tilde{p}(\zeta)$ is a parameter array over $\F$ 
for infinitely many values of $\zeta$.
Assume $\tilde{p}(\zeta)$ is the parameter array of a Leonard system
$\tilde{\Phi} (\zeta)$.
By Lemma \ref{lem:typeIindep} $\tilde{\Phi}(\zeta)$ has the same fundamental parameter
as $\Phi$.
By Lemmas \ref{lem:thcoincide} and \ref{lem:typeIa0ad} 
$\tilde{\Phi}(\zeta)$ has the same end-entries as $\Phi$.
\end{proofof}

\begin{proofof}{Theorem \ref{thm:main2}; case (i)}
Follows from Proposition \ref{prop:typeI}
and Lemmas \ref{lem:characterize}, \ref{lem:classify}.
\end{proofof}

\section{Proof of Theorem \ref{thm:main2}; part 2}
\label{sec:proof2(ii)}

In this section we prove Theorem \ref{thm:main2} for Case (ii).
Let $\Phi$ be a Leonard system of type IV. 
By Lemma \ref{lem:types}(v) $\Phi$ has diameter $d=3$.
Note that $\text{Char}(\F)=2$.
Let
$(\{\th_i\}_{i=0}^3, \{\th^*_i\}_{i=0}^3$, $\{\vphi_i\}_{i=1}^3, \{\phi_i\}_{i=1}^3)$
be the parameter array of $\Phi$.
Let $\{a_i\}_{i=0}^3$ (resp.\ $\{a^*_i\}_{i=0}^3$) be the principal sequence (resp.\ duall
principal sequence) of $\Phi$.

For a nonzero scalar $\zeta \in \F$ we define a sequence 
\[
  \tilde{p}(\zeta) =
  (\{\tth_i(\zeta)\}_{i=0}^3, \{\tth^*_i(\zeta)\}_{i=0}^3, 
\{\tvphi_i(\zeta) \}_{i=1}^3, \{\tphi_i(\zeta)\}_{i=1}^3)
\]
as follows.
Define
\begin{align*}
\tth_0 (\zeta)  &= \th_0, &
\tth_1 (\zeta) &= \th_0 + \frac{\zeta}{a^*_0-\th^*_0}, &
\tth_2 (\zeta) &= \th_3 + \frac{\zeta}{a^*_0-\th^*_0}, &
\tth_3 (\zeta) &= \th_3,
\\
\tth^*_0 (\zeta) &= \th^*_0, &
\tth^*_1 (\zeta) &= \th^*_0 + \frac{\zeta}{a_0-\th_0}, &
\tth^*_2 (\zeta) &= \th^*_3 + \frac{\zeta}{a_0-\th_0}, &
\tth^*_3 (\zeta) &= \th^*_3,
\end{align*}
and define
\begin{align*}
 \tvphi_1 (\zeta) &= \zeta, 
\\
 \tvphi_2(\zeta) &=
   \left( \th^*_0 - \th^*_3 + \frac{\zeta}{a_0-\th_0} \right)
   \left( \th_0 - \th_3 + \frac{\zeta}{a^*_0-\th^*_0} \right), 
\\
 \tvphi_3 (\zeta) &= \frac{a_3 - \th_3}{a_0-\th_0} \, \zeta,
\\
 \tphi_1 (\zeta) &= \frac{a_0-\th_3}{a_0-\th_0} \, \zeta
\\
 \tphi_2 (\zeta) &= 
   \left( \th^*_0 - \th^*_3 + \frac{\zeta}{a_0-\th_0}  \right)
   \left( \th_0 - \th_3 + \frac{\zeta}{a^*_0-\th^*_0} \right),  
\\
 \tphi_3 (\zeta) &= \frac{a_3-\th_0}{a_0-\th_0} \, \zeta. 
\end{align*}
We have defined $\tilde{p}(\zeta)$.
We prove the following result:

\begin{proposition}   \label{prop:typeIV}   \samepage
\ifDRAFT {\rm prop:typeIV}. \fi
The sequence $\tilde{p}(\zeta)$ is the parameter array  over $\F$
for infinitely many values of $\zeta$.
Assume $\tilde{p}(\zeta)$ is the parameter array of a Leonard system
$\tilde{\Phi}(\zeta)$.
Then $\tilde{\Phi}(\zeta)$ has the same fundamental parameter and the
same end-entries as $\Phi$.
\end{proposition}

The following three lemmas can be routinely verified.

\begin{lemma}   \label{lem:typeIVthithj}   \samepage
\ifDRAFT {\rm lem:typeIVthithj}. \fi
We have
\begin{align*}
\tth_0 (\zeta) - \tth_1 (\zeta) &= \frac{\zeta}{a^*_0-\th^*_0}, 
&
\tth_0 (\zeta) - \tth_2 (\zeta)  &= \th_0-\th_3 + \frac{\zeta}{a^*_0-\th^*_0}, 
\\
\tth_0 (\zeta) -\tth_3 (\zeta) &= \th_0-\th_3,
&
\tth_1 (\zeta) - \tth_2(\zeta) &= \th_0-\th_3, 
\\
\tth_1 (\zeta) - \tth_3(\zeta) &= \th_0-\th_3 + \frac{\zeta}{a^*_0-\th^*_0}, 
&
\tth_2 (\zeta) - \tth_3(\zeta) &= \frac{\zeta}{a^*_0-\th^*_0},
\\
\tth^*_0 (\zeta) - \tth^*_1(\zeta) &= \frac{\zeta}{a_0-\th_0},
&
\tth^*_0 (\zeta) - \tth^*_2(\zeta) &= \th^*_0-\th^*_3 + \frac{\zeta}{a_0-\th_0}, 
\\
\tth^*_0 (\zeta) - \tth^*_3(\zeta) &= \th^*_0-\th^*_3,
&
\tth^*_1 (\zeta) - \tth^*_2(\zeta) &= \th^*_0-\th^*_3, 
\\
\tth^*_1 (\zeta) - \tth^*_3(\zeta) &= \th^*_0-\th^*_3 + \frac{\zeta}{a_0-\th_0},
 &
\tth^*_2 (\zeta) - \tth^*_3(\zeta) &= \frac{\zeta}{a_0-\th_0}.
\end{align*}
\end{lemma}

\begin{lemma}   \label{lem:typeIVindep}    \samepage
\ifDRAFT {\rm lem:typeIVindep}. \fi
Each of the expressions 
\[
  \frac{\tth_0 (\zeta) - \tth_3 (\zeta)}{\tth_1 (\zeta) - \tth_2(\zeta)},     \qquad\qquad
  \frac{\tth^*_0 (\zeta) - \tth^*_3(\zeta) }{\tth^*_1(\zeta) - \tth^*_2(\zeta)}
\]
is equal to $1$.
\end{lemma}

For $1 \leq i \leq 3$ define
\[
 \vth_i = \sum_{\ell=0}^{i-1}
   \frac{\tth_\ell (\zeta) - \tth_{d-\ell} (\zeta)}
          {\tth_0 (\zeta) - \tth_d (\zeta)}.
\]

\begin{lemma}    \label{lem:typeIVvth}   \samepage
\ifDRAFT {\rm lem:typeIVvth}. \fi
We have
\begin{align*}
 \vth_1 &= 1, & \vth_2 &=0, & \vth_3 &= 1.
\end{align*}
\end{lemma}

\begin{lemma}   \label{lem:typeIVtvphitphi}   \samepage
\ifDRAFT {\rm lem:typeIVtvphitphi}. \fi
For $1 \leq i \leq 3$
\begin{align*}
  \tvphi_i (\zeta) &= 
    \tphi_1(\zeta) \vth_i + (\tth^*_i (\zeta) - \tth^*_0(\zeta) )(\tth_{i-1}(\zeta) - \tth_3(\zeta)),
\\
 \tphi_i (\zeta) &= 
   \tvphi_1(\zeta) \vth_i + (\tth^*_i (\zeta) - \tth^*_0(\zeta) )(\tth_{3-i+1} (\zeta) - \tth_0(\zeta) ).
\end{align*}
\end{lemma}

\begin{proof}
Routine verification using Lemmas \ref{lem:typeIVthithj},  \ref{lem:typeIVvth}
and \eqref{eq:a0th0/adth0}.
\end{proof}

\begin{lemma}    \label{lem:typeIVtparray}     \samepage
\ifDRAFT {\rm lem:typeIVtprray}. \fi
The sequence $\tilde{p}(\zeta)$ is a parameter array over $\F$
for infinitely many values of $\zeta$.  
\end{lemma}

\begin{proof}
We verify conditions (i)--(v) in Lemma \ref{lem:classify}.
Conditions (iii) and (iv) are satisfied by  Lemma \ref{lem:typeIVtvphitphi}.
Condition (v) is satisfied by Lemma \ref{lem:typeIVindep}.
Note that $\F$ has infinitely many elements since $\F$ is algebraically closed.
Observe that for $1 \leq i \leq 3$,
$\tvphi_i (\zeta)$ and $\tphi_i (\zeta)$ vanish for at most two values 
of $\zeta$.
Observe by Lemma \ref{lem:typeIVthithj} that 
for $0 \leq i < j \leq 3$,
$\tth_i (\zeta) - \tth_j (\zeta)$ and $\tth^*_i (\zeta) - \tth^*_j (\zeta)$
vanish for at most one value of $\zeta$.
Thus conditions (i) and (ii) are satisfied for infinitely many values of $\zeta$.
The result follows.
\end{proof}

\begin{lemma}    \label{lem:typeIVa0ad}   \samepage
\ifDRAFT {\rm lem:typeVIa0ad}. \fi
Assume $\tilde{p}(\zeta)$ is a parameter array of a Leonard system $\tilde{\Phi}(\zeta)$.
Let $\{\ta_i (\zeta)\}_{i=0}^3$ (resp.\ $\{\ta^*_i (\zeta) \}_{i=0}^3$) be the principal
sequence (resp.\ dual principal sequence) of $\tilde{\Phi} (\zeta)$.
Then
\[
 \ta_0 (\zeta) = a_0,  \qquad
 \ta_3 (\zeta) = a_3,  \qquad
 \ta^*_0 (\zeta) = a^*_0, \qquad
 \ta^*_3 (\zeta) = a^*_3.
\]
\end{lemma}

\begin{proof}
Similar to the proof of Lemma \ref{lem:typeIa0ad}.
\end{proof}

\begin{proofof}{Proposition \ref{prop:typeIV}}
By Lemma \ref{lem:typeIVtparray} $\tilde{p}(\zeta)$ is a parameter array over $\F$
for infinitely many values of $\zeta$.
Assume $\tilde{p}(\zeta)$ is the parameter array of a Leonard system
$\tilde{\Phi}(\zeta)$.
By Lemma \ref{lem:typeIVindep} $\tilde{\Phi}(\zeta)$ has fundamental parameter $0$.
By the construction and Lemma \ref{lem:typeIVa0ad}
$\tilde{\Phi}(\zeta)$ has the same end-entries as $\Phi$.
\end{proofof}

\begin{proofof}{Theorem \ref{thm:main2}; case (ii)}
Follows from Proposition \ref{prop:typeIV}
and Lemmas \ref{lem:characterize}, \ref{lem:classify}.
\end{proofof}

\bigskip
\bigskip\noindent
{\Large \bf
Acknowledgement}

\medskip \noindent
The author would like to thank Paul Terwilliger for giving this paper a close reading
and offering many valuable suggestions.

\bigskip
\bigskip\noindent
{\Large \bf
Appendix A}

\medskip \noindent
Let $\Phi$ be a Leonard system over $\F$ with diameter $d \geq 3$.
Let 
\[
(\{\th_i\}_{i=0}^d, \{\th^*_i\}_{i=0}^d, \{\vphi_i\}_{i=1}^d, \{\phi_i\}_{i=1}^d)
\]
be the parameter array of $\Phi$, and let $\beta$ be the fundamental
parameter of $\Phi$.
Let $\{a_i\}_{i=0}^d$ (resp.\ $\{a^*_i\}_{i=0}^d$) be the principal parameter
(resp.\ dual principal parameter) of $\Phi$.
Let the scalar $\Delta$ be from \eqref{eq:Delta}:
\[
  \Delta = (a_0 - \th_0)(a^*_0-\th^*_d) - (a_d-\th_0)(a^*_0-\th^*_0).
\]
Assume $\vphi_1+\vphi_d \neq \phi_1 + \phi_d$.
Note that $\Delta \neq 0$ by Lemma \ref{lem:Delta0}.
By Proposition \ref{prop:Delta0}
the type of $\Phi$ is one of I, II, III$^+$, III$^-$.
For each type, we display formulas that represent the parameter array
in terms of $\beta$ and the end-entries.

\bigskip\noindent
{\bf Type I}.

\bigskip

Pick a nonzero $q \in \F$ such that $\beta = q + q^{-1}$.

\bigskip

\noindent
For $0 \leq i \leq d$ define
\[
 K_i =
  \frac{(q^i-1)(\th_0-\th_d)}
         {(q^{d-1}-1)(q^d-1) \Delta^*} \, L_i,
\]
where
\[
 L_i = (q^{2d-i-1}-1)(a_0-\th_0)(a^*_d-\th^*_0) - q^{d-i}(q^{i-1}-1)(a_0-\th_d)(a^*_0-\th^*_0).
\]
Then for $0 \leq i \leq d$
\begin{align*}
 \th_i &= \th_0 + K_i,   &
 \th^*_i &= \th^*_0 + K^*_i,
\end{align*}
and for $1 \leq i \leq d$
{\small
\begin{align*}
 \vphi_i &= K^{\Downarrow}_{d-i+1} K^*_i 
      - \frac{(q^i-1)(q^{d-i+1}-1)(q^{d-1}+1)(a_0-\th_d)(a_d-\th_0)(a^*_0-\th^*_0)(\th^*_0-\th^*_d)}
                {(q^d-1)^2 \Delta},
\\
 \phi_i &=  K_{d-i+1} K^*_i
      - \frac{(q^i-1)(q^{d-i+1}-1)(q^{d-1}+1)(a_0-\th_0)(a_d-\th_0)(a^*_0-\th^*_0)(\th^*_0-\th^*_d)}
               {(q^d-1)^2 \Delta}.
\end{align*}
}

\bigskip\noindent
{\bf Type II}. 

\bigskip

\noindent
For $0 \leq i \leq d$ define
\[
 K_i =  \frac{ i (\th_0-\th_d)}
                 {d(d-1) \Delta^*} \, L_i,
\]
where
\[
 L_i = (2d-i-1)(a_0-\th_0)(a^*_d-\th^*_0) - (i-1)(a_0-\th_d)(a^*_0-\th^*_0).
\]
Then for $0 \leq i \leq d$
\begin{align*}
 \th_i &= \th_0 + K_i,   &
 \th^*_i &= \th^*_0 + K^*_i,
\end{align*}
and for $1 \leq i \leq d$
\begin{align*}
 \vphi_i &= K^{\Downarrow}_{d-i+1} K^*_i 
      - \frac{2 i (d-i+1) (a_0-\th_d)(a_d-\th_0)(a^*_0-\th^*_0)(\th^*_0-\th^*_d)}
                {d^2  \Delta},
\\
 \phi_i &=  K_{d-i+1} K^*_i
      - \frac{2 i (d-i+1) (a_0-\th_0)(a_d-\th_0)(a^*_0-\th^*_0)(\th^*_0-\th^*_d)}
               {d^2 \Delta}.
\end{align*}

\bigskip\noindent
{\bf Type III$^+$}. 

\bigskip

\noindent
For $0 \leq i \leq d$ define
\[
 K_i =
  \begin{cases}
   \displaystyle - \frac{ i (\th_0-\th_d)}{d}    & \text{ if $i$ is even},
  \\ 
   \rule{0mm}{6mm}
   \displaystyle 
   \;\; \frac{ \th_0-\th_d}
           {d \Delta^*} \, L_i    & \text{ if $i$ is odd},
   \end{cases}
\]
where
\[
  L_i = (2d-i-1)(a_0-\th_0)(a^*_d-\th^*_0) + (i-1)(a_0-\th_d)(a^*_0-\th^*_0).
\]
Then for $0 \leq i \leq d$
\begin{align*}
 \th_i &= \th_0 + K_i,   &
 \th^*_i &= \th^*_0 + K^*_i,
\end{align*}
and for $1 \leq i \leq d$
\begin{align*}
 \vphi_i &=
  \begin{cases}
   K^{\Downarrow}_{d-i+1} K^*_i 
      - \frac{2 i (d-1) (a_0-\th_d)(a_d-\th_0)(a^*_0-\th^*_0)(\th^*_0-\th^*_d)}
                {d^2  \Delta}
                                          &  \text{ if $i$ is even},
   \\
   K^{\Downarrow}_{d-i+1} K^*_i 
      - \frac{2 (d-i+1) (d-1) (a_0-\th_d)(a_d-\th_0)(a^*_0-\th^*_0)(\th^*_0-\th^*_d)}
                {d^2  \Delta}
                                          &  \text{ if $i$ is odd},
  \end{cases}
\\
 \phi_i &=  
  \begin{cases}
   K_{d-i+1} K^*_i
      - \frac{2 i (d-1) (a_0-\th_0)(a_d-\th_0)(a^*_0-\th^*_0)(\th^*_0-\th^*_d)}
               {d^2 \Delta}
                                       & \text{ if $i$ is even},
  \\
   K_{d-i+1} K^*_i
      - \frac{2 (d-i+1)(d-1) (a_0-\th_0)(a_d-\th_0)(a^*_0-\th^*_0)(\th^*_0-\th^*_d)}
               {d^2 \Delta}
                                       & \text{ if $i$ is odd}.
  \end{cases}
\end{align*}

\bigskip\noindent
{\bf Type III$^-$}. 

\bigskip

\noindent
For $0 \leq i \leq d$ define
\[
 K_i =  \frac{\th_0-\th_d}{(d-1) \Delta^*} \, L_i
\]
where
\[
 L_i =    
  \begin{cases}
     i (a_0-\th_0)(a^*_d-\th^*_0) + i (a_0-\th_d)(a^*_0-\th^*_0)
     & \text{ if $i$ is even},
  \\  
     (2d-i-1)(a_0-\th_0)(a^*_d-\th^*_0) - (i-1)(a_0-\th_d)(a^*_0-\th^*_0)
                                          & \text{ if $i$ is odd}.
  \end{cases}
\]
Then for $0 \leq i \leq d$
\begin{align*}
 \th_i &= \th_0 + K_i,   &
 \th^*_i &= \th^*_0 + K^*_i,
\end{align*}
and for $1 \leq i \leq d$
\begin{align*}
 \vphi_i &=
  \begin{cases}
   K^{\Downarrow}_{d-i+1} K^*_i                &  \text{ if $i$ is even},
   \\
   K^{\Downarrow}_{d-i+1} K^*_i 
      - \frac{2 (a_0-\th_d)(a_d-\th_0)(a^*_0-\th^*_0)(\th^*_0-\th^*_d)}
                {\Delta}
                                          &  \text{ if $i$ is odd},
  \end{cases}
\\
 \phi_i &=  
  \begin{cases}
   K_{d-i+1} K^*_i                   & \text{ if $i$ is even},
  \\
   K_{d-i+1} K^*_i
      - \frac{2 (a_0-\th_0)(a_d-\th_0)(a^*_0-\th^*_0)(\th^*_0-\th^*_d)}
               { \Delta}
                                       & \text{ if $i$ is odd}.
  \end{cases}
\end{align*}

\bigskip\bigskip\noindent
Kazumasa Nomura\\
Tokyo Medical and Dental University\\
Kohnodai, Ichikawa, 272-0827 Japan\\
email: knomura@pop11.odn.ne.jp

\medskip\noindent
{\small
{\bf Keywords.} Leonard pair, tridiagonal pair, tridiagonal matrix.
\\
\noindent
{\bf 2010 Mathematics Subject Classification.} 05E35, 05E30, 33C45, 33D45
}

\end{document}